\documentclass{conm-p-l}

\usepackage{amsmath,amsthm,amssymb}
\usepackage{mathtools} 
\usepackage{stmaryrd} 
\usepackage{mathrsfs}

\usepackage{tikz}
\usepackage[all]{xy}

\usepackage{hyperref} 

\newtheorem{theorem}{Theorem}[section]
\newtheorem{proposition}[theorem]{Proposition}
\newtheorem{lemma}[theorem]{Lemma}

\theoremstyle{definition}
\newtheorem{definition}[theorem]{Definition}

\theoremstyle{remark}
\newtheorem{remark}[theorem]{Remark}

\def\N{\mathbb{N}}
\def\Z{\mathbb{Z}}
\def\Q{\mathbb{Q}}
\def\R{\mathbb{R}}
\def\C{\mathbb{C}}
\def\scrG{\mathcal{G}}
\def\scrA{\mathscr{A}}

\DeclareMathOperator\id{id}
\DeclareMathOperator\rk{rk}
\DeclareMathOperator\pr{pr}

\DeclareMathOperator\Hom{Hom}

\DeclareMathOperator\Ad{Ad}
\DeclareMathOperator\ad{ad}

\usepackage[english]{babel}
\useshorthands{"}
\defineshorthand{"-}{\nobreakdash-\hspace{0pt}}

\begin{document}

\title[Connections on central extensions]
{Connections on central extensions, lifting gerbes,\\
and finite-dimensional obstruction vanishing}

\author[I. Biswas]{Indranil Biswas}

\address{School of Mathematics, Tata Institute of Fundamental Research,
Homi Bhabha Road, Mumbai 400005}

\email{indranil@math.tifr.res.in}

\thanks{The first author is partially supported by a J.~C.~Bose Fellowship.}

\author[M. Upmeier]{Markus Upmeier}

\address{The Mathematical Institute, Radcliffe
Observatory Quarter, Woodstock Road, Oxford, OX2 6GG, U.K.}

\email{upmeier@maths.ox.ac.uk}

\thanks{The second author thanks the TIFR, Mumbai, for hospitality during 03/2017 and was partially funded by DFG grant UP~85/2-1 of the priority program SPP~2026 Geometry at Infinity.}

\date{November 17, 2019}



\begin{abstract}
Given a central extension of Lie groups, we study the classification
problem of lifting the structure group together with a given connection.
For reductive structure groups we introduce a connective structure
on the lifting gerbe associated to this problem.
Our main result classifies all connections on the central extension
of a given principal bundle. In particular, we find that admissible
connections are in one-to-one correspondence with parallel
trivializations of the lifting gerbe.
Moreover, we prove a vanishing result for
Neeb's obstruction classes for finite-dimensional Lie groups.
\end{abstract}

\maketitle

\setcounter{tocdepth}{1}\tableofcontents

\section*{Introduction}

Consider a central extension of Lie groups
$1 \longrightarrow A \longrightarrow B \longrightarrow
C \longrightarrow 1.$
Given a principal $C$"-bundle ${\pi\,\colon\, P \,\longrightarrow\, X}$, it is a classical problem to seek principal $B$"-bundles ${\rho\,\colon R\,\longrightarrow\, X}$ equipped with an isomorphism
${\sigma \,\colon\, R/A\,\longrightarrow\,P}$ of the structure group reduction to $B$; see Grothendieck~\cite{grothendieck1957quelques, grothendieck1958general} for a solution of the obstruction and classification problem in terms of sheaf cohomology. From a
more geometric point of view one associates a lifting gerbe $\scrG_\pi$ to
this setup, see Murray~\cite{MR2681698} or Brylinski~\cite{Bry}; the 
trivializations of $\scrG_\pi$ (if they exist) are in bijective correspondence with the
central extensions ${\rho\,\colon R\,\longrightarrow\, X}$ of the bundle ${\pi\,\colon\, P \,\longrightarrow\, X}$ up to
an isomorphism.

Now suppose that the principal $C$-bundle $P$ is equipped also with a connection $\omega^\pi$.
There being no further topological obstructions, our main result
describes all connections on the central extension ${\rho\,\colon R\,\longrightarrow\, X}$ in Theorem~\ref{cor12341},
under the assumption that $B$ is a reductive Lie group.
The $A$"-bundle
${\sigma\,\colon\, R\,\longrightarrow\, P}$ is naturally $B$"-equivariant
and before proving Theorem~\ref{cor12341} we
therefore need to develop some basic theory for equivariant connections
in Section~\ref{sec:distortion}. In particular we introduce the \emph{distortion} of a $G$"-connection and establish some elementary results
for the distortion and for quotient connections; see
Propositions~\ref{quotient-exact} and \ref{ExtKES}. This is later applied
to ${G=B}$, though it comes at hardly any extra cost to establish Theorem~\ref{cor12341}
in a fully equivariant context, with a Lie group $G$ acting on the base manifold $X$.

Restricting to admissible connections \eqref{compatible-connection}, we then
apply Theorem~\ref{cor12341} to solve the extension of the lifting problem involving
connections. Using $\omega^\pi$ we introduce in Proposition~\ref{obsExt}
a connection on the lifting gerbe $\scrG_\pi$ and explain in Theorem~\ref{thm:main2} that parallel trivializations of $\scrG_\pi$ correspond
to central extensions equipped with admissible connections. This result
is equivalent to Theorem~\ref{thm:main1}, stated without reference to
gerbes.

Using crossed modules, Neeb associates in \cite{MR2208114}
a cohomology class $H^3(X,\,\C)$ to every central extension
with $A=\C^*$. In the final section we shall address the question,
raised by Neeb \cite[Problem~VI]{MR2208114}, that asks which obstruction
classes may occur. We find in Theorem~\ref{thm:obstructions}
that for finite-dimensional Lie groups $C$ these obstruction
classes always vanish. This is proven using classical algebraic
topology for Lie groups, mostly due to Borel~\cite{MR0051508} and gives a different approach to that of Brylinski \cite[Chapter~6]{Bry}. In infinite dimensions, the
situation is drastically different; see Remark~\ref{rem:infinite-dim}.

\section{Preliminaries}

\subsection{Infinitesimal actions}

Let $A$ be a Lie group with Lie algebra $\mathfrak{a}$.
Let $P\times A\,\longrightarrow\, P$ be a smooth action of $A$ on a manifold $P$. Then we denote the linearized
action by juxtaposition $${\cdot\, \colon\, TP \times TA \,\longrightarrow\, TP}\, .$$ If we regard $A\subset TA$ and 
$P\subset TP$ embedded as the zero section, then we may write
\[
X\cdot \xi = X\cdot a + p\cdot \xi\qquad\forall ~ X\in T_pP,\ \xi \in T_aA\, .
\]
Differentiating the condition for an action gives
\begin{equation}
(X\cdot \xi)\cdot \eta \,=\, X\cdot (\xi\cdot \eta)\qquad \forall ~ X\in T_pP, \xi \in T_aA,
\ \eta\in T_bA\, .
\end{equation}
In particular, this notation applies to $P=A$ acting on itself by right multiplication and
similarly to left actions of $A$.

\subsection{Connections on principal bundles}

By convention, all our principal bundles are smooth and have their structure group 
acting from the right; sometimes the term ``principal bundle'' will be shortened to ``bundle''.

\begin{definition}
Let $\pi\,\colon\, P\,\longrightarrow\, X$ be a principal $A$-bundle. A \emph{connection} $H^{\pi}$ on $\pi$ is an 
$A$"-equivariant complementary subbundle $H^\pi\,\subset\, TP$ of the vertical bundle:
\begin{align*}
	H^\pi\oplus \ker(d\pi) &= TP,
	&H^\pi_{pa}&=H^\pi_p a
	&\forall~ a\in A,\ p\in P\, .
\end{align*}
Equivalently, a connection is a $1$"-form $\omega^\pi \in \Omega^1(P,\mathfrak{a})$ satisfying
\begin{align}\label{connection1form-axioms}
\omega^\pi(p\cdot \xi) &= \xi,
&\omega^\pi(v\cdot a) &= a^{-1}\omega(v)a
& \forall~ p\in P,\, v\in TP,\, \xi \in \mathfrak{a},\, a\in A\, .
\end{align}
\end{definition}

The \emph{horizontal subspace} $H^\pi$ is simply the kernel of $\omega^\pi$, the \emph{vertical subspace}
\begin{equation}\label{vertical-subbundle}
	T^\mathrm{vert}_p P = \ker \pi_{*,p}= p\cdot \mathfrak{a},
\end{equation}
giving an identification of vector bundles $T^\mathrm{vert}P\,=\,
P\times\mathfrak{a}$. The adjoint bundle
\[
 \ad(P) = T^\mathrm{vert}P / A = P\times^A \mathfrak{a},
\]
is associated to $P$ using the adjoint action of $A$ on $\mathfrak{a}$.

Here we recall that the
bundle $P\times^A V \,\longrightarrow\, X$ \emph{associated} to an $A$-space $V$ (also called the \emph{Borel construction}) is the quotient of
$P\times V$ by the twisted diagonal action of $A$; in particular,
the points of $\{(pa,
a^{-1}v)\}_{a\in A}$ are identified in the quotient
$P\times^A V$. When $V=B$ is a group and the action is the left
translation using a homomorphism $A\,\longrightarrow\, B$, then $P\times^A B$ becomes a principal $B$-bundle.

\begin{proposition}\label{fact:pullback-connections}
Given a homomorphism\/ $(F,f)$ of\/ $A$"-bundles as in
\begin{equation}\begin{aligned}
\xymatrix{
P\ar[d]_{\pi}\ar[r]^F&Q\ar[d]^{\sigma}\\
X\ar[r]_f& Y
}
\end{aligned}\end{equation}
we may pull back connections from\/ $\sigma$ to\/ $\pi$ by\/ $H^{\pi}\,\coloneqq \,F_*^{-1}H^{\sigma}$ or, equivalently, by\/ $\omega^{\pi}\coloneqq F^*\omega^\sigma$.
\end{proposition}

\subsection{Equivariant bundles and connections}

Let $G$ be another Lie group acting smoothly on a manifold $X$.
The space of $G$"-invariant differential forms on $X$ is denoted
$\Omega_G^*(X)$. 

\begin{definition}
A \emph{$G$"-equivariant} principal $A$-bundle is an $A$-bundle ${\pi\colon P \to X}$ together with a 
$G$"-action $G\times P \,\longrightarrow\, P$ satisfying
\begin{align*}
g(pa)&=\,(gp)a,
&\pi(gp)&=\,g\pi(p)
&\forall~ g\in G,\, p\in P,\, a\in A\, .
\end{align*}
\end{definition}

\begin{definition}
A connection $H^\pi\,\subset\, TP$ on a  $G$-equivariant principal $A$-bundle ${\pi\,\colon\, P \longrightarrow X}$
is a \emph{$G$"-connection} if for all $g\in G$ we have
\[
H^\pi_{gp} \,=\, gH^\pi_p\, .
\]
Equivalently the connection $1$"-form $\omega^\pi$ is $G$"-invariant.
\end{definition}

Then the horizontal lifting map $P\times_X TX \,\longrightarrow\, TP$ is $G$"-equivariant. We write $\overline{v}|_p$ 
for the \emph{horizontal lift} of $v \in TX$ to $p\in P$. From the
definitions it follows that
\begin{align}
\overline{gv}|_{gp}&=g\cdot\overline{v}|_p,
&\overline{v}|_{pa}&=\overline{v}|_p\cdot a
&\forall~ p\in P,\, g\in G,\, a\in A\,.
\end{align}

\begin{definition}
$\mathscr{A}^G(\pi)$ is the \emph{space of $G$"-connections} on $\pi$. When
$G=\{1\}$ we simply write $\mathscr{A}(\pi)$ for the space of connections on $\pi$.
\end{definition}

\begin{remark}
The space $\mathscr{A}^G(\pi)$ may be empty. For example, if $A$ is a Lie subgroup
of a Lie group $G$ and $\pi$ is the principal $A$-bundle $G\longrightarrow G/A$, consider the
left-translation actions of $G$ on $G$ and $G/A$. Then there is no
$G$-connection if the short exact sequence
$$
0 \,\longrightarrow\, \mathfrak{a} \,\longrightarrow\, \operatorname{Lie}(G) \,\longrightarrow\,
\operatorname{Lie}(G)/\mathfrak{a} \,\longrightarrow\, 0
$$
does not admit an $A$-equivariant splitting. For example, when when $G=
\text{SL}(2,{\mathbb C})$ and $A$ is the Borel subgroup of it defined by lower
triangular matrices, then there is no $G$-connection on the principal $A$-bundle
$G\longrightarrow G/A$.
\end{remark}

If $\mathscr{A}^G(\pi)$ is non-empty, then 
it is an affine space modelled on the vector space $\Omega^1_G(P;\, \mathfrak{a})^A=
\Omega^1_G(X;\, \mathrm{ad}(P))$, the space of $G$"-invariant $1$"-forms on $X$ with values in the adjoint bundle.

\section{Distortion of Equivariant Connections}\label{sec:distortion}

In this section, assume that $G$ acts freely and properly on $X$,
so that the quotient map
\[
\sigma\,\colon\, X\,\longrightarrow\,\underline{X}\,:=\, G\backslash X
\]
is a principal $G$"-bundle over the manifold $\underline{X}$.
Let ${\pi\,\colon\, P\,\longrightarrow\, X}$ be a $G$"-equivariant principal $A$"-bundle, and let ${\underline{P}\,=\,G\backslash P}$. Then the quotient map ${\underline{\pi}\,\colon\,
\underline{P}\,\longrightarrow\, \underline{X}}$ is a principal $A$"-bundle and we have a commutative diagram
\begin{equation}\label{quotient-bundle}\begin{aligned}
\xymatrix{
P\ar[r]^{\pi}\ar[d]_{\overline\sigma} &X\ar[d]^{\sigma}\\
\underline{P}\ar[r]_{\underline{\pi}}&\underline{X}.
}
\end{aligned}\end{equation}

\subsection{Quotients of connections}

\begin{proposition}\label{lemma:quotient-connection}
In the situation of \eqref{quotient-bundle}, given a connection\/ $\omega^\sigma$ on the principal\/ $G$-bundle\/ ${\sigma\,\colon\,X\,\longrightarrow\,\underline{X}}$ and a\/ $G$"-connection $\omega^\pi$ on
the principal\/ $A$-bundle\/ $P\,\longrightarrow\, X$, we 
get a \emph{quotient connection}\/ $\omega^{\underline{\pi}}$ on the quotient bundle\/ $\underline{\pi}\,\colon\, 
\underline{P}\,\longrightarrow\,\underline{X}$ by setting
\begin{equation}\label{eqn:quotient-connection}
	H^{\underline{\pi}} \coloneqq \overline\sigma_*\left(
	H^{\pi}\cap \pi_*^{-1}H^{\sigma}
	\right).
\end{equation}
\end{proposition}

\begin{proof}
Since $G$ acts freely on ${H^{\pi}\cap \pi_*^{-1}H^{\sigma}\subset TP}$,
it follows that \eqref{eqn:quotient-connection} is a subbundle of $T\underline{P}$. It is easy to check that it is both 
$G$"-equivariant and $A$"-equivariant. Also
\begin{align*}
&{\overline\sigma}_*^{-1}\left(\underline\pi_*^{-1}(0) \cap H^{\underline{\pi}}\right)\\
&=\pi_*^{-1}\sigma_*^{-1}(0)\cap \overline\sigma_*^{-1}\overline\sigma_*(H^{\pi}
\cap \pi_*^{-1}H^{\sigma}) &\text{see~\,\eqref{quotient-bundle}, \eqref{eqn:quotient-connection}}\\
&\subset \pi_*^{-1}\sigma_*^{-1}(0)\cap H^{\pi}\cap \pi_*^{-1}H^{\sigma}\\
&= \pi_*^{-1}(\sigma_*^{-1}(0) \cap H^\sigma) \cap H^\pi\\
&= \pi_*^{-1}(0) \cap H^\pi = 0&\text{(horizontal)}
\end{align*}
This shows that $H^{\underline\pi}$ has trivial intersection with $\underline\pi^{-1}(0)$. The rank is
\[
\rk H^{\underline\pi} = \rk \underline{\pi}_*\overline{\sigma}_*(H^\pi\cap \pi^{-1}_* H^\sigma)
= \rk \sigma_*\pi_*(H^\pi\cap \pi^{-1}_* H^\sigma)
= \rk \sigma_*H^\sigma=\rk T\underline{X},
\]
using the fiberwise injectivity of $\underline{\pi}_*$ on $H^{\underline\pi}$ just shown; similarly for $\pi_*, \sigma_*$. This proves that \eqref{eqn:quotient-connection} is a complementary subbundle of the vertical bundle $\underline\pi^{-1}_*(0)$.
\end{proof}

\begin{remark}
The quotient connection \eqref{eqn:quotient-connection} clearly depends on $\omega^\sigma$,
a point that is sometimes problematic in the literature.
\end{remark}

Using $\overline\sigma\,\colon\, P\,\longrightarrow\, \underline{P}$ and the pullback connections of Proposition~\ref{fact:pullback-connections} we obtain 
from a connection on $\underline\pi$ a $G$"-connection on $\pi$. This determines a map
\begin{equation}\label{eqn:r}
	r\,\colon\, \scrA(\underline\pi)\,\longrightarrow\, \scrA^G(\pi)\, .
\end{equation}
In particular, the set of $G$"-connections on $\pi$ is non-empty. Also, for each fixed choice of connection on $\sigma$, we get from Proposition~\ref{lemma:quotient-connection} a quotient connection map
\begin{equation}\label{eqn:q}
q(\omega^\sigma)\,\colon \,\scrA^G(\pi)\,\longrightarrow\, \scrA(\underline\pi)\, .
\end{equation}

\begin{proposition}\label{section-retraction}
Taking pullback connections \eqref{eqn:r} is a right inverse
to taking quotients connection \eqref{eqn:q}, so
\/
$q(\omega^\sigma)\circ r = {\rm Id}_{\scrA(\underline\pi)}$. In particular,\/
$q(\omega^\sigma)$ is surjective and\/ $r$ is injective.
\end{proposition}

\begin{proof}
Putting $H^\pi\,=\,\overline{\sigma}_*^{-1}H^{\underline\pi}$ from Proposition~\ref{fact:pullback-connections} into 
\eqref{eqn:quotient-connection} gives:
\begin{align*}
&\overline\sigma_*\left(
	\overline{\sigma}_*^{-1}H^{\underline\pi}\cap \pi_*^{-1}H^{\sigma}
	\right)\\
&=\overline\sigma_*\left( \overline\sigma_*^{-1}H^{\underline\pi}\cap \overline\sigma_*^{-1}\overline\sigma_* \pi_*^{-1}H^{\sigma} \right) & (\overline\sigma_*^{-1}\overline\sigma_*A=A+\ker\overline\sigma_*)\\
&=H^{\underline{\pi}} \cap \overline\sigma_*\pi_*^{-1}H^{\sigma} &(\text{$\overline\sigma_*$ surjective})\\
&\subset H^{\underline{\pi}} \cap \underline{\pi}^{-1}_*\underline{\pi}_*\overline\sigma_*\pi_*^{-1}H^{\sigma}\\
&= H^{\underline{\pi}}\cap \underline{\pi}^{-1}_*\sigma_*\pi_*\pi_*^{-1}H^{\sigma}
&\text{see~\,\eqref{quotient-bundle}}\\
&= H^{\underline{\pi}}\cap \underline{\pi}^{-1}_*\sigma_*H^{\sigma} &(\text{$\pi_*$ surjective})\\
&= H^{\underline{\pi}}\cap \underline{\pi}^{-1}_*T\underline{X} &(\text{horizontal})\\
&= H^{\underline{\pi}}. & (\text{$\underline\pi_*$ surjective})
\end{align*}
This implies $q(\omega^\sigma)\circ r = {\rm Id}_{\scrA(\underline\pi)}$.
\end{proof}

\subsection{Distortion}

We denote the Lie algebra of $G$ by $\mathfrak{g}$.

Let $\omega^{\pi}$ be a 
$G$"-connection on $\pi$. If we restrict $P$ to a $G$"-orbit of $X$ we have a 
canonical connection given by the $G$"-action on $P$, because the action of $G$
on $X$ is free. This gives us a canonical family of 
connections parametrized by the orbit space $\underline{X}$,
and so determines a gauge potential with respect to $\omega^{\pi}$,
which we call the distortion of the $G$"-connection.

\begin{definition}\label{d1}
Let $G$ act freely and properly on $X$ and let
${\pi\,\colon\, P\,\longrightarrow\, X}$ be a $G$"-equivariant
principal $A$"-bundle.
The \emph{$G$"-distortion} of a $G$"-connection $\omega^{\pi}$ on $P$ is the following $G$"-equivariant vector bundle homomorphism over $X$:
\begin{align}\label{eqn:distortion}
	\tau(\omega^{\pi})\colon \mathfrak{g}\times X &\longrightarrow\, \ad(P),
	&\tau(\xi,x)&\coloneqq \overline{\xi\cdot x}|_p - \xi\cdot p
	&\forall~ x\in X,\, p\in \pi^{-1}(x)
\end{align}
Here the trivial bundle $\mathfrak{g}\times X$ gets the $G$"-action $g(\xi,x)=(g\xi g^{-1},gx)$ and the adjoint bundle $\ad(P)=(P\times\mathfrak{a})/A$ is equipped with the
$G$"-action $g(p,\eta)\coloneqq (gp,\eta)$.
\end{definition}

To explain the distortion is well-defined, note that
the right hand side of \eqref{eqn:distortion} is vertical. If we replace $p$ by $pa$ we get
\begin{align*}
\overline{\xi\cdot x}|_{pa} - \xi\cdot p \cdot a = \left(\overline{\xi\cdot x}|_{p} - \xi\cdot p\right) \cdot a,
\end{align*}
which represents the same vector in $\ad(P)=T^\mathrm{vert}P/A$. The
$G$"-equivariance property is verified by the calculation
$\overline{g\xi g^{-1}\cdot gx}|_{gp} - g\xi g^{-1} gp = g\overline{\xi x}|_p -
g\xi\cdot p.$

\begin{remark}
Alternatively, the $G$"-distortion can be identified with the composition
\[
	\mathfrak{g} \times P \,\longrightarrow\, TP \,\stackrel{-\omega^\pi}{\longrightarrow}\, 
\mathfrak{a}\, ,
\]
where the first map is the infinitesimal action. Thus $-\omega^\pi(\xi\cdot p)= \tau(\omega^\pi)(\xi,x)|_p$. For $A$ abelian, $\ad(P)\,=\,X\times\mathfrak{a}$, so the distortion may then be regarded  as a homomorphism $ \mathfrak{g}\times_G X\,\longrightarrow\, \mathfrak{a}$.
\end{remark}

The distortion measures how far a $G$"-equivariant connection is
from being a pullback:

\begin{proposition}\label{quotient-exact}
Let\/ $G$ act freely and properly on\/ $X$ and
let\/ $\pi\,\colon\, P\,\longrightarrow\, X$ be a\/ $G$"-equivariant principal\/ $A$"-bundle. Then the pullback map \eqref{eqn:r} is part
of an exact sequence
\begin{equation}\label{first-exact}
0 \longrightarrow \scrA(\underline{\pi}) \xrightarrow{\enskip r\enskip} \scrA^G(\pi) \xrightarrow{\enskip\tau\enskip} \Hom_G\big(T^\mathrm{vert}X, \ad(P)\big) \longrightarrow 0.
\end{equation}
A connection\/ $\omega^\sigma$ on\/ ${\sigma\,\colon\, X\,\longrightarrow\,\underline{X}}$ together with a base-point in\/ $\scrA(\underline\pi)$
determines a split of \eqref{first-exact}, in which case we have a dual split short exact sequence
\begin{equation}\label{sec-exact}
0 \longrightarrow \Hom_G(T^\mathrm{vert}X, \ad(P)) \longrightarrow \scrA^G(\pi) \xrightarrow{\;q(\omega^\sigma)\;} \scrA(\underline{\pi}) \longrightarrow 0\, ,
\end{equation}
using the quotient connection map \eqref{eqn:q} associated to\/ $\omega^\sigma$.
\end{proposition}

Recall here that a sequence of affine spaces
\[
0\longrightarrow U\xrightarrow{\enskip f\enskip} V \xrightarrow{\enskip g\enskip} W \longrightarrow 0
\]
is short exact if $g\circ f$ is a constant map $w_0$ and if for any
base-point $u_0\in U$ and $v_0=f(u_0)$ the corresponding sequence of vector spaces is 
exact (the image of $g\circ f$ is taken as the base-point in $W$). This is equivalent to $g\circ f=\operatorname{const}_{w_0}$, $f$ injective, 
$g$ surjective, and $f(U)=g^{-1}(w_0)$. In particular, the definition is independent 
of the choice of $u_0$.

\begin{proof}[{Proof of Proposition \textup{\ref{quotient-exact}}}]
Clearly $\tau\circ r=0$ and $r$ is injective by Proposition~\ref{section-retraction}. 
It is enough to construct an affine linear map
\[
r'\,\colon\, \Hom_G(X\times\mathfrak{g}, \ad(P))\,\longrightarrow\, \scrA^G(\pi)
\]
satisfying $\tau\circ r'=1$. For this we pick a connection on $X\,\longrightarrow\, \underline{X}$ and fix an element 
$\omega^{\underline\pi}\in \scrA(\underline\pi)$. Then $\scrA^G(\pi)=\Omega^1_G(X;\ad(P))+\omega^\pi$ for $\omega^\pi \coloneqq r(\omega^{\underline\pi})$. Using the connection on $X$ we may extend maps on $T^\mathrm{vert}X$ by zero on
$H^{\sigma}$:
\[
r'\colon \Hom_G(X\times\mathfrak{g}, \ad(P)) \xrightarrow{\text{extend}} \Hom_G(TX, \ad(P))\xrightarrow{+\omega^\pi} \scrA^G(\pi)\, .
\]
Using $\tau(\omega^\pi)=\tau(r(\omega^{\underline\pi}))=0$ we get $\tau\circ r'=1$.
\end{proof}

\section{Central Extensions}

\subsection{Connections on Reductions}

Consider a Lie group central extension
\begin{equation}\label{eqn:central-ext}
1\,\longrightarrow\, A \,\stackrel{\alpha}{\longrightarrow}\, B \,\stackrel{\beta}{\longrightarrow}\,
C \,\longrightarrow\, 1.
\end{equation}
The Lie algebras of $B$ and $C$ will be denoted by $\mathfrak b$ and $\mathfrak c$ 
respectively. Let $G$ be a Lie group and $\gamma\,\colon\, G\,\longrightarrow\, B$ a group homomorphism (mostly 
the trivial homomorphism). By composing with the adjoint representation we get 
an induced $G$"-action on $\mathfrak b$. Similarly, using $\beta\circ \gamma$
we get an action of $G$ on $\mathfrak c$ for which the homomorphism
$\beta_*\,\colon\, \mathfrak{b}\,\longrightarrow\, \mathfrak{c}$ becomes
$G$"-equivariant.

\begin{proposition}\label{ExtKES}
Let\/ $G$ be a Lie group acting on a manifold\/ $P$,\/ ${\gamma\,\colon\, G \,\longrightarrow\, B}$ a homomorphism,
and\/ ${\sigma\, \colon\, R \, \longrightarrow \, P}$ a\/ $G$"-equivariant\/
$A$"-bundle. Define a $G$"-action on\/ $R\times^A B$ by\/
$g[r,b] \coloneqq [gr,\gamma(g)b]$ for\/ $[r,b]\in R\times^A B$ and\/ $g \in G$. Then we have an exact sequence
\begin{equation}\label{EQI-exact}
	0\longrightarrow \scrA^G(\sigma) \longrightarrow \scrA^G(R\times^A B) \xrightarrow{\enskip q\enskip}
\Omega^1_{G}(P;\mathfrak{c}),
\end{equation}
using\/ $G$-equivariant\/ $1$"-forms\/ $TP\,\longrightarrow\, \mathfrak{c}$
and where the map\/ $q$ is constructed in \eqref{def-q} below. 
If\/ $\beta_*$ admits a\/ $G$"-equivariant section\/ $\mathfrak{c}\,\longrightarrow\,\mathfrak{b}$, then\/ $q$ is surjective and a choice
of section canonically determines a split of the sequence \eqref{EQI-exact}.
\end{proposition}

In particular, the sequence always splits when $G$ is compact (for example $G={1}$), or when $G=B$, $\gamma=\id_B$ and $B$ is a reductive Lie group, by which we mean that every real $B$-representation is completely reducible.

\begin{proof}[Proof of Proposition \textup{\ref{ExtKES}}]
On the level of tangent bundles, the principal $A$"-bundle ${\hat\sigma \colon R\times B\longrightarrow R\times^A B}$ induces an exact sequence
\[
	0 \longrightarrow \ad(\hat\sigma) \longrightarrow \frac{TR\times TB}{A} \longrightarrow T(R\times^A B) \longrightarrow 0,
\]
we see that connection $1$"-forms $\omega$ on $R\times^A B$ correspond via
\begin{equation}\label{omegaomega0}
	\omega(v+\xi) = b^{-1}\omega_0(v)b + b^{-1}\xi\qquad
	\forall~ v\in T_rR,\, \xi \in T_bB
\end{equation}
to $1$-forms $\omega_0 \in \Omega^1(R;\mathfrak{b})$ satisfying
\begin{align}\label{formsScalarExt}
\omega_0(r\eta)&=\eta,
&\omega_0(va) &= a^{-1}\omega_0(v)a
&\forall ~ r\in R,\, v\in T_rR,\, a\in A,\, \eta\in \mathfrak{a}.
\end{align}
Moreover $\omega$ is $G$"-invariant precisely when $\omega_0\,\colon\, TR \,\longrightarrow\, \mathfrak{b}$ is 
$G$"-equivariant. 
Hence $\scrA^G(\sigma)\subset \scrA^G(R\times^A B)$ is the subspace of $\mathfrak{a}$"-valued forms. We may define
\begin{equation}\label{def-q}
	\sigma^*q(\omega) \coloneqq \omega_0 \mod \mathfrak{a},
\end{equation}
since by \eqref{formsScalarExt} the form $\omega_0$ is horizontal 
and $A$"-invariant, and $\scrA^G(\sigma)$ is the kernel of $q$.

Finally, suppose we have a $G$"-equivariant section of $\beta_*$. Then we find a $G$"-equivariant bundle homomorphism ${R\times^A \mathfrak{c}
\,\longrightarrow\, R\times^A \mathfrak{b}}$ that is a section of ${R\times^A \beta_*}$. Let $\omega^\sigma \in 
\scrA^G(\sigma)\subset 
\scrA^G(R\times^AB)$. Then we may identify $\scrA^G(R\times^AB)$ with $\Omega^1_G(P;\mathfrak{b})$ and a section of $q$ is provided by
\[
	\Hom_G(TP,\mathfrak{c}) \xrightarrow{\enskip s_*\;} \Hom_G(TP, \mathfrak{b}) \xrightarrow{\;+\omega^\sigma\,} \scrA^G(R\times^AB).
\]
This also shows surjectivity of $q$.
\end{proof}

\subsection{Central lifting problems}

Throughout, let $G$ be a Lie group acting on a manifold $X$ and
fix a central extension of Lie groups as in \eqref{eqn:central-ext}.
Given a $C$"-bundle $P$, we are interested in the lifting problem of
whether the structure group may be reduced along $\beta$ as follows:

\begin{definition}\label{central-extension}
Let $\pi \colon P \longrightarrow X$ be a $G$"-equivariant principal $C$"-bundle.
An \emph{equivariant central extension} of $\pi \colon P \longrightarrow X$ is a pair
$(\rho,\sigma)$ of a $G$"-equivariant principal $B$"-bundle
\[
	\rho \colon R \longrightarrow X
\]
and a $G$"-equivariant $B$"-homomorphism
\[
	\sigma\colon R \longrightarrow P.
\]
(here we use $\beta$ to convert the $C$"-action on $P$ into a $B$"-action.)
\end{definition}

In particular, the principal $C$"-bundles $R/A$ and $P$ are isomorphic via $\sigma$.
The data in Definition~\ref{central-extension} will be
conveniently denoted by ${\rho\,\colon\, R\xrightarrow{\enskip\sigma\enskip} P\xrightarrow{\enskip\pi\enskip} X}.$

\begin{remark}
This terminology also makes sense in the topological category for a central extension of topological groups \eqref{eqn:central-ext}
in which $\beta$ is an $A$-bundle (automatic when $B,C$ are Lie groups).
When $C$ is a Lie group, there is no difference between the topological and the
smooth category; this follows from the two facts that $B$ inherits a unique smooth
structure and every principal bundle for a Lie group has a unique smooth
structure (see M\"uller--Wockel \cite{MR2574141}).
Similarly, in the holomorphic category we can suppose \eqref{eqn:central-ext} is a central extension of complex Lie groups, meaning that $\beta$ is a holomorphic $A$"-bundle.
\end{remark}

\begin{remark}
There is also the weaker equivariance assumption that for each $g\in G$ we have $g^*P \cong P$. It may be of interest to study the relationship between these two versions of the problem, both in the smooth and the holomorphic category.
\end{remark}

We are interested in putting connections on central extensions. In principle,
connections on $R$ and on $P$ are independent, so we impose the following
compatibility condition for a connection on $R$ to qualify as a central extension
of a given connection on $P$:

\begin{definition}
Let $\rho\colon R\xrightarrow{\sigma} P\xrightarrow{\pi} X$ be a central extension as in Definition~\ref{central-extension}, so that we are given a $G$"-equivariant principal $B$"-bundle $\rho \colon R \longrightarrow X$ and a $G$"-equivariant $B$"-homomorphism $\sigma\colon R \longrightarrow P$.
Suppose ${\pi\colon P \longrightarrow X}$ is equipped with a connection $\omega^\pi$. Then a connection $\omega^\rho$ on $\rho$
is \emph{admissible} if
\begin{equation}\label{compatible-connection}
	\beta_*(\omega^\rho) = \sigma^* \omega^\pi.
\end{equation}
\end{definition}

\subsection{Special case of free actions}\label{ssec:free-actions}

In this section we briefly discuss the simple case when the
action of $G$ on $X$ is free. Then one may understand central
extensions from the point of view of sheaf cohomology. The
freeness implies that the
equivariant central extensions correspond bijectively to central extensions of the
quotient principal $C$"-bundle ${G\backslash P \longrightarrow G\backslash X}$. The pullback of the quotient bundle 
along the 
projection ${X \longrightarrow G\backslash X}$ is isomorphic to $P$, and similarly for $R$. The central extension 
induces an exact sequence
\begin{equation}\label{es1}
	\cdots \longrightarrow H^1(G\backslash X,\, B) \longrightarrow H^1(G\backslash X,\,C) \xrightarrow{\enskip\delta\enskip} H^2(G\backslash X,\, A).
\end{equation}

Of course, isomorphism classes of $C$-bundles on $G\backslash X$ correspond bijectively
to classes in $H^1(G\backslash X, \,C)$. We have thus shown:

\begin{proposition}\label{prop:free}
Let\/ $G$ act freely on\/ $X$, and let\/ ${P \,\longrightarrow\, X}$ be a $G$"-equivariant
$C$"-bundle.
There exists an equivariant central extension of\/ $P$ if and only if\/
${\delta[G\backslash P]=0}$, where\/ $\delta$ is the coboundary homomorphism in
\eqref{es1}.

If\/ $A,B,C$ are abelian then the space of equivariant central extensions
up to an isomorphism, if non-empty, is an affine space modelled on the quotient group\/
${H^1(G\backslash X, A)\big/H^0(G\backslash X,C)}$.
\end{proposition}

When $B$ is not abelian, the problem of determining which obstruction classes can occur was raised in \cite[Problem~VI]{MR2208114}.

\begin{remark}
 One may follow a similar approach in the equivariant case using Grothendieck equivariant sheaves.
 The corresponding sheaf cohomology $H^1_G(X, \,C)$ is related through acyclic $G$-covers of $X$ to $C$-bundles (see 
\cite[p.~211]{grothendieck1957quelques}).
 Here a $G$-cover $\{U_i\}_{i\in I}$ is an open cover with a fixed-point free $G$-action on $I$ so
 that $gU_i = U_{gi}$. It is acyclic when $H^k(U_{i_1\cdots i_n}, \,C)\,=\,0$, $k>0$, for ordinary sheaf cohomology.
\end{remark}

\subsection{Topological classification}

We first explain how to classify central extensions of
Definition~\ref{central-extension} in the smooth category, without connections.
Continue to assume that $G$ is a Lie group acting on $X$ and
fix a central extension of Lie groups as in \eqref{eqn:central-ext}.

\begin{definition}\label{dfn:difference-bundle}
Let $\pi\colon P \longrightarrow X$ be a $G$"-equivariant $C$"-bundle. The
\emph{difference map} of $P$ is
\[
	\delta_P\,\colon\, P \times_X P \,\longrightarrow\, C,\quad
	z_1\delta_P(z_1,z_2) := z_2.
\]
The map $\delta_P$ is $G$"-invariant for the action of $G$ by
${g(z_1,z_2)=(gz_1,gz_2)}$ on ${P\times_X P}$.
Pullback of $\beta$ along $\delta_p$
defines a $G$-equivariant $A$"-bundle $Q_P$, called the \emph{difference bundle} of $P$:
\begin{equation}
\label{eqn:diff-bundle}
\begin{aligned}\xymatrix{
Q_P\ar[r]\ar[d]&B\ar[d]^\beta\\
P\times_XP \ar[r]_-{\delta_P} & C
}\end{aligned}
\end{equation}
The product in $B$ determines the \emph{multiplication map}
\begin{equation}\label{eqn:multiplication}
	m\colon Q_P \times_P Q_P \longrightarrow Q_P,
	\quad
	m\big((p_1,p_2,b),(p_2,p_3, b')\big) = (p_1,p_3,bb').
\end{equation}
\end{definition}

The data of Definition~\ref{dfn:difference-bundle} is a $G$"-equivariant version of the lifting $A$"-gerbe in~\cite{MR2681698}, see also Proposition~\ref{obsExt}, which is trivial (in the sense of the following proposition)
if and only if a solution to the central lifting problem of Definition~\ref{central-extension} may be found. This result is well known, but we
wish to fix notation and give proof that is particularly convenient
for studying connections.

\begin{proposition}\label{propo0}
Let\/ ${\pi\colon P \longrightarrow X}$ be a\/ $G$"-equivariant\/ $C$"-bundle.
Let\/ $Q_P$ be the difference bundle of\/ $P$ with its multiplication map\/ $m$.
Up to an isomorphism,\/ $G$"-equivariant central extensions\/
${\rho\colon R\xrightarrow{\sigma} P\xrightarrow{\pi} X}$
of\/ $\pi$ correspond to pairs\/
 $(\sigma, \psi)$ consisting of a\/ $G$"-equivariant\/ $A$"-bundle\/ $\sigma\colon R\longrightarrow P$ and a\/ $G$"-equivariant\/ $A$"-homomorphism
\begin{equation}\label{psi}
	\psi\colon R\times^A R \longrightarrow Q_P
\end{equation}
over\/ $P\times_X P$ satisfying the cocycle identity
\begin{equation}\label{cocycle}
	m \big(\psi[r_1,r_2], \psi[r_2,r_3] \big) = \psi[r_1,r_3] \qquad \forall(r_1,r_2,r_3)
\in R\times_X R \times_X R.
\end{equation}
\end{proposition}

\begin{remark}\label{propo0rem}\mbox{}
\begin{enumerate}
\item[(i)] In the Borel construction in \eqref{psi} the right factor of $R$ is equipped with the 
left $A$"-action through inversion. So in right hand notation $[r_1,r_2]=[r_1a,r_2a]$ 
for $a\in A$ and the fiber $(R\times^A R)_{(p_1, p_2)}$ is the set of $A$-maps 
$R_{p_1}\,\longrightarrow\, R_{p_2}$.

\item[(ii)]
Composing $\psi$ with the map ${Q_P \longrightarrow B}$ from \eqref{eqn:diff-bundle}
we may equivalently view $\psi$ as a $G$"-invariant map $R\times_X R \,\longrightarrow\, B$ (satisfying some additional properties)
for which equation \eqref{cocycle} means $\psi(r_1,r_2)\cdot \psi(r_2,r_3)=\psi(r_1,r_3)$.
\end{enumerate}
\end{remark}

\begin{proof}[{Proof of Proposition~\textup{\ref{propo0}}}]
Suppose first that $\rho\colon R\xrightarrow{\sigma} P\xrightarrow{\pi} X$ is a solution of the  extension problem. 
Then we get the difference map $\psi\,\colon \,R\times_X R\,\longrightarrow\,B$ with $r_1\psi(r_1,r_2)=r_2$. Since $A$ 
is central, $\psi$ factors through the diagonal action of $A$ giving \eqref{psi} satisfying \eqref{cocycle}. 

Conversely, suppose an $A$"-bundle
$\sigma\,\colon\, R\,\longrightarrow\, P$ and
$\psi$ as in \eqref{psi} are given satisfying \eqref{cocycle}.
We must construct a $B$"-action on $R$.
For this we first define a $C$"-action on the bundle ${\xi\,\colon\, R\times^A B \,\longrightarrow\, 
P}$, where $\xi[r,b]=\sigma(r)$. Given $[r_1,b_1] \in R \mathbin{\times^A} B$ and $c\in C$, choose $b_c\in B$ with $\beta(b_c)=c$.
Setting
\[
p_1\coloneqq \sigma(r_1)\, ,\qquad p_2\coloneqq p_1c\, ,
\]
the element $(p_1, p_2, b_c)\in Q_P$ corresponds under 
$\psi$ to a pair $(r_1,r_2)$, normalized under the $A$"-action by the choice of $r_1$. Define the action of $c \in C$ by
\begin{align}\label{Caction}
	[r_1,b_1]\cdot c&\coloneqq [r_2, b_c^{-1}b_1].
\end{align}
This definition is independent of the choice of lift $b_c$.
We have $$\xi([r_1,b_1]\cdot c) = \xi[r_1,b_1]\cdot c,$$ since $p_1\cdot c=p_2$, and the right $C$"-action on $R\times^AB$ commutes with the obvious right $B$"-action. Hence $\xi$ is a $C$"-equivariant $B$-bundle. The quotient
\begin{equation}\label{quotient}
\begin{aligned}\xymatrix{
R\times^A B\ar[r]\ar[d]_\xi&(R\times^A B)/C\ar[d]^{\xi/C}\\
P\ar[r]^-\pi&X=P/C.
}\end{aligned}\end{equation}
is a $B$"-bundle $\xi/C$ that can be identified with $\rho\,\colon\, R\,\longrightarrow\, X$ via the $A$"-map
\begin{equation}\label{Rquotient}
 	R \,\longrightarrow\, (R\times^A B)/C,\quad
	r \,\longmapsto\, [r,\,1].
\end{equation}
This allows us to transport the $B$-operation from $(R\times^A B)/C$ to $R$.
Under the identification $\sigma(r) = \xi[r,1]$ and so $\sigma$ is a $B$"-homomorphism, as required in Definition~\ref{central-extension}.
\end{proof}

\begin{remark}
When $A$ is central and $R$ has a $B$"-action, there is an additional $B$-action on $R\times_A B$ given by ${[r_1,b_1]\star b \coloneqq [r_1b,b^{-1}b_1]}$
which is trivial on $A$. In this case $R\times^AB$ becomes a $C$"-equivariant $B$"-bundle.
\end{remark}

\subsection{Connections on central extensions}

With the topological classification out of the way,
we now study connections on central extensions
of principal $C$"-bundles.
For this we shall assume the existence of an
$\Ad|_B$"-equivariant section $s$ of ${\beta_*\colon \mathfrak{b}\longrightarrow \mathfrak{c}}$. Let
\begin{equation}\label{split}
0\longrightarrow \mathfrak{c}\xrightarrow{\enskip s\enskip} \mathfrak{b} \xrightarrow{\enskip t\enskip} \mathfrak{a} \longrightarrow 0
\end{equation}
be the dual short exact sequence, characterized by the formula
$1=\alpha t+s\beta$. Then $t$ is also $\Ad|_B$"-invariant.
We may regard $s$ as a connection on the bundle ${\beta\colon B \longrightarrow C}$.
This yields pullback connections $\omega^{Q_P}$ in \eqref{eqn:diff-bundle} on all of the
difference bundles $Q_P$ of Definition~\ref{dfn:difference-bundle}.

\begin{theorem}\label{thm:main1}
Assume an\/ $\Ad|_B$"-equivariant splitting \eqref{split} of a central extension of Lie groups as in \eqref{eqn:central-ext}. Let\/ $G$ be a Lie group acting on a manifold\/ $X$ and let\/ ${\pi\,\colon\, P \,\longrightarrow\, X}$
be a\/ $G$"-equivariant\/ $C$"-bundle equipped with\/ $G$"-equivariant
connection\/ $\omega^\pi$. Let\/ $Q_P$ denote the difference bundle of\/ $P$
with its pullback connection\/ $\omega^{Q_P}$.
Up to an isomorphism, every\/ $G$"-equivariant central extension\/
${\rho\,\colon\, R\xrightarrow{\enskip\sigma\enskip} P\xrightarrow{\enskip\pi\enskip} X}$ of\/ $\pi$
with admissible\/ $G$"-equivariant
connection\/ $\omega^\rho$ corresponds to a triple\/ $(\sigma,\omega^\sigma,\psi)$
consisting of a\/ $G$"-equivariant\/ $A$"-bundle
\[
	\sigma\colon R\longrightarrow P
\]
with\/ $G$"-equivariant connection\/ $\omega^\sigma$ and a\/ $G$"-equivariant\/ $A$-homomorphism
\[
\psi\colon R\times^A R \longrightarrow Q_P
\]
over\/ $P\times_X P$ satisfying the cocycle identity \eqref{cocycle} for\/ $\psi$ and also the equation
\begin{equation}
\pr_2^*\omega^\sigma - \pr_1^*\omega^\sigma = \psi^*\omega^{Q_P}.
\end{equation}
\end{theorem}

Theorem~\ref{thm:main1} will be a direct consequence of the more general
Theorem~\ref{cor12341} below. For greater clarity, we shall restrict to
$G=\{1\}$, but the arguments apply in general.
Given a central extension $\rho\colon R\xrightarrow{\sigma} P\xrightarrow{\pi} X$, consider from \eqref{quotient} and \eqref{Rquotient} the diagram
\[\xymatrix{
R\times^A B\ar[r]^-{\xi}\ar[d]_{\overline\pi}&P\ar[d]^\pi\\
R\ar[r]_\rho&X.
}\]
This exhibits $\rho$ as the $C$"-quotient of $\xi$. Here $\overline\pi[p,b]=pb$. We may therefore apply Propositions~\ref{quotient-exact}~and~\ref{ExtKES} (for $\gamma=\id_B$) to get the following exact sequences:
\[\xymatrix{
0\ar[r] & \scrA(R\to X) \ar[r]& \scrA^C(R\times^A B \to P)\ar@{=}[d] \ar[r]^-\tau& \Hom_C(P\times\mathfrak{c},  \mathfrak{b}) \ar[r]& 0\\
0 \ar[r]& \scrA^B(R\to P) \ar[r]& \scrA^B(R\times^A B \to P) \ar[r]^-q& \Omega^1_{B}(P;\mathfrak{c}) \ar[r]& 0\\
}\]
Once we fix a connection on $P\,\longrightarrow\, X$, both sequences split (but recall 
that the split depends on a choice of a base-point, in $\scrA(R\to X)$
and $\scrA^B(R\to P)$ respectively). We describe the following compositions of maps in the diagram:
\begin{align*}
\scrA(R\to X) &\longrightarrow \Omega^1_{B}(P;\mathfrak{c}),\\
\scrA^B(R\to P) &\longrightarrow \Hom_C(P\times\mathfrak{c}, \mathfrak{b}).
\end{align*}
The first map assigns to $\omega^\rho$ the form $\kappa$ defined by
\begin{align*}
	\kappa(v)& \coloneqq \omega^\rho(\hat{v})\mod\mathfrak{a} &\forall~ v\in T_pP,
\,\sigma_*\hat{v}=v\, .
\end{align*}
For the second map, the image of $\omega^\sigma \in \scrA^B(R\to P)$ can be computed from its distortion by
\begin{align*}
	(p,\xi) &\mapsto \tau(\omega^\sigma)(p,\hat\xi) + \hat\xi &\forall ~
\xi \in \mathfrak{c},\, \beta_*\hat\xi = \xi\, .
\end{align*}
Moreover, the connection $H^\pi$ on $P\,\longrightarrow\, X$ determines an obvious map
\[
	\scrA^B(R\to P) \,\longrightarrow\, \scrA(R\to X),\qquad
	H^{\rho}_r \coloneqq \sigma_{*,r}^{-1}\left( H^{\pi}_{\sigma(r)}\right)\cap H^\sigma_r,
\]
which is also obtained from the diagram by using the split. Finally, use the $\Ad|_B$-equivariant section $s$ of 
$\beta_*$ (this is simply a $B$"-biequivariant connection on the $A$"-bundle $\beta\,\colon\, B\,\longrightarrow\, C$, 
where $B$ acts by multiplication on either side and similarly on $C$ via $\beta$)
to get \[
	\scrA(R\to X) \,\cong\, \ker\Big(
	\scrA^B(R\to P) \oplus \Omega^1_{B}(P;\mathfrak{c})
	\,
	\longrightarrow
	\,
	\Hom_C(P\times\mathfrak{c}, \mathfrak{b})
	\Big).
\]

It is useful to make this isomorphism explicit:

\begin{proposition}\label{prop-cc}
There is a bijection
\begin{equation}
	\scrA(R\to X) \overset{\cong}{\,\longrightarrow\,}
	\ker\Big( \scrA^B(R\to P) \xrightarrow{\enskip T\enskip} \Hom(P\times_C \mathfrak{b}, \mathfrak{a}) \Big) \oplus \scrA(P\to X)
    \label{prop-cc-eqn}
\end{equation}
Here\/ $T(\omega_A)\,=\,\tau + t\circ\pr_2$ for\/ $\omega^\sigma \,\in\, \scrA^B(R\to P)$ is essentially the distortion of\/ $\omega^\sigma$, namely
\[
	T(\omega_A)(p,\xi_B) = t(\xi_B) - \omega_A(r\cdot \xi_B),\qquad \forall~ r\in R: \sigma(r)=p, \xi_B \in \mathfrak{b}.
\]
\end{proposition}

\begin{proof}
$\scrA(R\to X)$ consists of connection $1$"-forms $\omega^\rho \,\in \,\Omega^1(R;\mathfrak{b})$ satisfying
\begin{enumerate}
\item
$\omega^\rho\colon TR \,\longrightarrow\, \mathfrak{b}$ is $B$-equivariant,
\item
$\omega^\rho(r\cdot \xi_B)=\xi_B$ for all $\xi_B \in \mathfrak{b}$.
\end{enumerate}
The right hand side of \eqref{prop-cc-eqn} is bijective to all
pairs $(\omega^\sigma, \hat\omega^\rho) \in \Omega^1(R;\mathfrak{a}) \times \Omega^1(R;\mathfrak{c})$ satisfying:
\begin{enumerate}
\item
$\omega^\sigma$ $B$-invariant
\item
$\omega^\sigma(r\cdot \xi_A)=\xi_A$ for all $\xi_A \in \mathfrak{a}$
\item
$T(\omega^\sigma)=0$
\item
$\hat\omega^\rho\colon TR \longrightarrow \mathfrak{c}$ is $B$-equivariant
\item
$\hat\omega^\rho(r\cdot\xi_A) = 0$ for all $\xi_A \in \mathfrak{a}$
\item
$\hat\omega^\rho(p\cdot \xi_C)=\xi_C$ for all $\xi_C \in \mathfrak{c}$.
\end{enumerate}
This is because from iv) and v)~we see that $\hat\omega^\rho=\sigma^*\omega^\pi$ is the pullback of a unique form $\omega^\pi$ in $\Omega^1(P;\mathfrak{c})$. In this notation, the bijection \eqref{prop-cc-eqn} is then given by
\begin{align*}
\omega^\sigma &\coloneqq t(\omega^\rho),
&\hat\omega^\rho &\coloneqq \beta(\omega^\rho),\\
\omega^\rho &\coloneqq \alpha\circ \omega^\sigma + s \circ \hat\omega^\rho.
\end{align*}
It is easy to check that these formulas define maps that are inverse to each other. Also properties i), ii) correspond to properties i)--vi) under these bijections.
\end{proof}

\begin{proposition}\label{prop-cc2}
The kernel of\/ $\scrA^B(R\to P) \xrightarrow{\enskip T\enskip} \Hom(P\times_C \mathfrak{b}, \mathfrak{a})$ can be identified with the 
set of connections\/ $\omega^\sigma$ on\/ $\sigma\colon R\,\longrightarrow\, P$ with the additional property
\[
	\pr_2^*\omega^\sigma - \pr_1^*\omega^\sigma = \psi^*t(\theta_B) \in \Omega^1(R\times_X R; \mathfrak{a}),
\]
where\/ $\theta_B$ denotes the Maurer--Cartan form on\/ $B$.
\end{proposition}

Recall here that $\psi\,\colon\, R\times_X R \,\longrightarrow\, B$ encodes the $B$"-action on $R$, see Remark~\ref{propo0rem}.

\begin{proof}
The map $(\id_R, \psi)$ is an inverse diffeomorphism to the action
$$R\times B \,\longrightarrow\, R\times_X R\, ,\ \ (r,b)\longmapsto (r,rb)\, .$$ Taking derivatives, this means that 
any $(v_{r_1}, w_{r_2}) \in {T_{(r_1,r_2)}(R_1\times_X R_2)}$ may be expressed as
\begin{equation}
w_{r_2} = v_{r_1}\cdot b + r_1\cdot \xi_b\quad\text{for}\quad b\in B, \xi_b \in T_b B.
\end{equation}
For a $B$-invariant $\omega^\sigma$ in the kernel of $T$ we then compute at $(v_{r_1},w_{r_2})$
\begin{align*}
\omega^\sigma(w_{r_2}) - \omega^\sigma(v_{r_1}) &= \omega^\sigma(v_{r_1}\cdot b) + \omega^\sigma(r_1\cdot \xi_b) - \omega^\sigma(v_{r_1})\\
&= \omega^\sigma(r_1 b \cdot b^{-1}\xi_b) = b^{-1}t(\xi_b).
\end{align*}
The converse follows from the same computation
\[
 \omega^\sigma(v_{r_1}\cdot b) + \omega^\sigma(r_1\cdot \xi_b) - \omega^\sigma(v_{r_1}) = b^{-1}t(\xi_b)
\]
by setting $v_{r_1}=0$ it follows that $\omega^\sigma\in \ker(T)$ and $\xi_b=0$ proving $B$"-invariance.
\end{proof}

Combining Propositions~\ref{prop-cc} and \ref{prop-cc2} yields:

\begin{theorem}\label{cor12341}
Assume an\/ $\Ad|_B$"-equivariant splitting\/ $(s,r)$ of a central extension of Lie groups \eqref{eqn:central-ext} as in \eqref{split}. Let\/ $G$ be a Lie group acting on a manifold\/ $X$ and let
\[
\rho\colon R\xrightarrow{\enskip\sigma\enskip} P\xrightarrow{\enskip\pi\enskip} X
\]
be a\/ $G$"-equivariant central extension.
The space\/ $\scrA^G(R\to X)$ of $G$"-equivariant connections on $\rho$ is parameterized by pairs\/ $(\omega^\sigma, \omega^\pi)$ of\/ $G$"-equivariant connections on\/ $\sigma$ and\/ $\pi$ satisfying
\[
	\pr_2^*\omega^\sigma - \pr_1^*\omega^\sigma = \psi^*t(\theta_B) \in \Omega^1(R\times_X R; \mathfrak{a}),
\]
where\/ $\theta_B$ denotes the Maurer--Cartan form on\/ $B$ and where, as in Remark~\textup{\ref{propo0rem}}, the map\/ $\psi\,\colon\, R\times_X R \,\longrightarrow\, B$ encodes the\/ $B$"-action on\/ $R$.
\end{theorem}

\section{Lifting Gerbes and Connective Structures}

The goal of this section is to interpret the conditions appearing in
Theorem~\ref{thm:main1} in the terminology of gerbes.
Continue to assume an $\Ad|_B$"-equivariant splitting \eqref{split} of a central extension of Lie groups as in \eqref{eqn:central-ext}. Let $G$ be a Lie group acting on a manifold $X$ and let ${\pi\colon P \longrightarrow X}$ be a $G$"-equivariant $C$"-bundle equipped with $G$"-equivariant
connection $\omega^\pi$.

\subsection{Background on gerbes}

\begin{definition}
Let $A$ be an abelian Lie group. Let $\pi_i\,\colon\, P_i \,\longrightarrow\, X$ for $i=1,2$ be
principal $A$-bundles. The \emph{tensor product}
$P_1 \otimes_A P_2 \coloneqq (P_1 \times_X P_2) / A$
is another $A$-bundle over $X$, where the group $A$ acts by $(p_1,p_2)a=(p_1a,p_2a^{-1})$. Equivalently,
form the $(A\times A)$"-bundle $P_1\times_X P_2$ and take the associated $A$-bundle
using the multiplication 
homomorphism $A\times A \,\longrightarrow\, A$. Hence $P_1 \otimes_A P_2 = \Delta^*(P_1\times^A P_2)$ is the pullback of the Borel construction along the diagonal.
Given connections $\omega^{\pi_i}$ on $P_i$, the \emph{tensor product connection} is
\[
	\omega[v_{p_1},w_{p_2}] \coloneqq \omega^{\pi_1}(v_{p_1}) + \omega^{\pi_2}(v_{p_2}),\quad
	\forall~ [v_{p_1},w_{p_2}] \in T_{[p_1,p_2]} P_1\otimes_A P_2.
\]
\end{definition}

\begin{definition}
An \emph{$A$-gerbe} $\scrG$ on $X$ is given by a fiber bundle $Y\,\longrightarrow\, X$ together with a principal $A$"-bundle 
$Q\,\longrightarrow\, 
Y\times_X Y$ and a `multiplication' $A$-bundle homomorphism
\begin{equation}\label{eqn:defm}
m\colon \pr_{12}^*Q\otimes_A \pr_{23}^*Q \,\longrightarrow\, \pr_{13}^*Q,\quad
(f,g) \mapsto m(f,g)=g\circ f
\end{equation}
which is required to be associative $(h\circ g)\circ f=h\circ (g\circ f)$. A
gerbe is \emph{$G$"-equivariant} if $X, Y, Q$ are equipped with $G$"-actions
and all of the involved maps are $G$"-equivariant.
A \emph{connection} or \emph{connective structure} on a gerbe $\scrG$ is a connection on ${Q \longrightarrow
Y\times_X Y}$ for which \eqref{eqn:defm} is parallel. It is called
a \emph{$G$"-equivariant connection} if the $G$"-action on $Q$ is parallel.
\end{definition}

\begin{remark}
One pictures a gerbe $\scrG$ as a ``bundle of groupoids'', parameterized by $x\in X$. The category $\scrG_x$ has objects $\pi^{-1}(x)$. The morphism set in $\scrG_x$ from $y_1$ to $y_2$ is the fiber $\Hom_{\scrG_x}(y_1,y_2) = Q_{(y_1,y_2)}$. Composition, also sometimes called the multiplication in the groupoid, is induced by $m$, hence the terminology. The existence of identity morphisms and of inverses are a consequence of the invertibility of $m$,
see Murray~\cite{MR2681698}.
\end{remark}

\subsection{Lifting gerbes}

\begin{proposition}\label{obsExt}
Assume an\/ $\Ad|_B$"-equivariant splitting\/ $(s,r)$ of a central extension of Lie groups \eqref{eqn:central-ext} as in \eqref{split}.
Every central lifting problem as in Definition~\textup{\ref{central-extension}}
with fixed\/ $G$"-equivariant principal\/ $C$"-bundle\/
${\pi\colon P\longrightarrow X}$ defines a\/ $G$"-equivariant \emph{lifting $A$-gerbe}\/ $\scrG_\pi$ which is naturally equipped with a connective structure.
\end{proposition}

\begin{proof}
Using the notation of Definition~\ref{dfn:difference-bundle}, take
$Y=P$ and take for $Q$ the difference bundle $Q_P$ defined in
\eqref{eqn:diff-bundle}. 
The multiplication homomorphism
is given by \eqref{eqn:multiplication}.
The pullback connection\/ $\omega^{Q_P}$ on the difference bundle determines
a\/ $G$"-equivariant connection on\/ $\scrG_\pi$.
\end{proof}

\begin{remark}
For $A=\C^*$ it has been shown that conversely every gerbe is `stably' isomorphic to one given by a central lifting 
problem. Here we recall that as gerbes are a two-categorical notion, the usual notion of strict isomorphism should be weakened. The correct notion is that of stable isomorphism, for which we refer to Murray~\cite[Def.~5.3]{MR2681698}. For the proof one first shows that a gerbe is classified by its Dixmier--Douady class in 
$H^2(X,\,\underline{\C}^*_X) \,=\, H^3(X;\,\Z)$. \end{remark}

\begin{definition}
Let $\scrG=(Y\to X, Q,m)$ be a gerbe with connection. A \emph{parallel trivialization}
of $\scrG$ is an $A$"-bundle $\rho\colon R\longrightarrow X$ together with a connection-preserving isomorphism
${R\times^A R \longrightarrow Q}$ over $Y\times_X Y$, compatible with \eqref{eqn:defm}
in the obvious way, see Murray~\cite{MR2681698}.
When $\scrG$ is $G$"-equivariant, we require $R$ to have a $G$"-action and the isomorphism
to be $G$"-equivariant.
\end{definition}

With this terminology, we may reformulate Theorem~\ref{thm:main1} as follows:

\begin{theorem}\label{thm:main2}
Assume an\/ $\Ad|_B$"-equivariant splitting\/ $(s,r)$ of a central extension of Lie groups \eqref{eqn:central-ext} as in \eqref{split}.
Let\/ ${\pi\colon P \longrightarrow X}$ be a\/ $G$"-equivariant principal\/ $C$"-bundle.
Parallel $G$"-equivariant trivializations of the lifting gerbe\/ $\mathcal{G}_\pi$ of\/ $P$
are in one-to-one correspondence with central extensions\/ ${\rho\colon R\xrightarrow{\sigma} P\xrightarrow{\pi} X}$ with admissible\/ $G$"-equivariant connections.
\end{theorem}

\section{Vanishing of Neeb's Obstructions}

Consider a $C$"-bundle $P\,\longrightarrow\, X$ and a central extension of groups \eqref{eqn:central-ext} with 
$A\,=\,\C^*$.
Using crossed modules, Neeb associates in \cite{MR2208114} a cohomology class $H^3(X,\,\C)$ and shows that its 
vanishing is necessary for the existence of a central extension of $P$.
It was shown in \cite{gengoux} that up to torsion this is the full obstruction.
In this section we shall address the question, raised by Neeb \cite[Problem~VI]{MR2208114}, which obstruction classes may occur, in finite dimensions. This gives a different approach to an argument by Brylinski, see \cite[Chapter~6]{Bry}.

\begin{theorem}\label{thm:obstructions}
Let\/ $C$ be a finite-dimensional, connected Lie group. Then all of Neeb's obstruction classes for smooth principal\/ $C$"-bundles vanish.
\end{theorem}

For the proof we need two lemmas. Recall that by \cite{grothendieck1957quelques} the central extension \eqref{eqn:central-ext} induces an exact sequence in sheaf cohomology
\[
	\cdots\,\longrightarrow\, H^1(X,\,B)\,\longrightarrow\, H^1(X,\,C) \,\stackrel{\delta}{\longrightarrow}\, 
H^2(X,\,\C^*).
\]
The $C$"-bundle determines a class $[P] \in H^1(X,\,C)$. We will show that $\delta$ takes values in
the torsion subgroup of $H^2(X,\C^*)$. This suffices, since:

\begin{lemma}
Neeb's obstruction class is the image of\/ $\delta[P]$ in
\[
H^2(X,\,\C^*) \,\cong H^3(X;\, \Z)\,\longrightarrow\, H^3(X; \,\C)\, .
\]
\end{lemma}

All topological spaces are assumed to be paracompact.

\begin{definition}\cite[\S 14.4]{MR2456045}
Let $G$ be a topological group. A topological $G$"-principal bundle $EG \,\longrightarrow\, BG$ is \emph{universal} if
every free $G$"-space $E$ admits a $G$"-map $E\,\longrightarrow\, EG$, unique up to $G$"-homotopy. The base space $BG$ 
is the \emph{classifying space} of $G$.
\end{definition}

For any $G$ there exists a universal $G$"-bundle (e.g.~\cite[\S 14.4]{MR2456045}). It is unique up to
$G$"-homotopy equivalence. Applying the definition to the trivial bundle $S^n\times G$ we see that
all $[S^n, EG]$ are singletons. Hence $EG$ is (weakly) contractible.

The homotopy of $BG$ can be approximated be the following method.
Suppose $E\,\longrightarrow\, B$ is a $G$"-bundle with $n$"-connected total space $E$. Then the long
exact sequence of homotopy groups for a fibration shows that $B\,\longrightarrow\, BG$ is $(n+1)$"-connected. In
particular, $H_*(B;\,\Z)\,\cong\, H_*(BG;\,\Z)$ for $*\leq n$. Recall here that a topological space $X$ is \emph{$n$"-connected} if $\pi_0(X) = \pi_1(X) = \cdots = \pi_n(X)=\{0\}$. More generally, a continuous map $f\colon X \,\longrightarrow\, Y$ is said to be \emph{$n$"-connected} if $\pi_0(f), \ldots, \pi_{n-1}(f)$ are bijective and $\pi_n(f)$ is surjective.

\begin{lemma}\label{lem:all-torsion}
For all finite-dimensional connected Lie groups\/ $C$, the cohomology\/
$H^3(BC;\,\Z)$ is a torsion group.
\end{lemma}

\begin{proof}
By \cite[Theorem~6]{MR0029911} the group $C$ deformation retracts onto its maximal compact connected subgroup $G$. 
Using the long exact sequence of homotopy groups for fibrations, we see that $BG \,\longrightarrow\, BC$ is a weak 
equivalence, so we may assume $C\,=\,G$ is compact. We then have an embedding $\rho\,\colon\, G\,\longrightarrow\, 
O(n)$.

For any $N\in \N$ consider the Stiefel manifold $E=V_n(\R^N)$ with its free $G$"-action by $\rho$. The long exact sequence of homotopy groups for the fibration $O(k) \hookrightarrow O(k+1) \twoheadrightarrow S^k$ combined with $\pi_0(S^k)=\cdots=\pi_{k-1}(S^k)$ shows that $O(k) \to O(k+1)$ is $(k-1)$-connected. Iterating, we find $O(N-n) \to O(N)$ to be $(N-n-1)$-connected. Applying the long exact sequence of homotopy groups to the fibration $O(N-n) \hookrightarrow O(N) \twoheadrightarrow E$ of the Stiefel manifold implies that $E$ is $(N-n-1)$-connected.
Then $E\,\longrightarrow\, E/G\,=\,B$ is a $G$"-bundle with $N-n-1$-connected total space, so we may apply the remarks preceeding the lemma. Since $B$ is a compact manifold and  $N$ is arbitrary we conclude that all homology groups $H_n(BG;\,\Z)\,\cong\, H_n(B;\,\Z)$ are finitely generated.

According to Borel \cite{MR0051508}, the rational cohomology $H^*(BG;\,\Q)$ is a polynomial ring on even degree 
generators. Hence $H^3(BG;\,\Q)\,=\,0$ and so $H_3(BG;\Z)$ is a torsion group by the homological universal coefficient 
theorem.
Now apply the cohomological universal coefficient theorem
\[
	0 \longrightarrow \operatorname{Ext}^1_\Z(H_2(BG;\,\Z), \Z) \longrightarrow H^3(BG;\,\Z) \longrightarrow \Hom_\Z(H_3(BG;\,\Z), \Z) \longrightarrow 0.
\]
Putting $H_2(BG;\,\Z)\,=\,\Z^a \oplus T$ and $H_3(BG;\,\Z)\,=\,T'$ for the torsion subgroups $T,\, T'$ we see 
$H^3(BG;\,\Z)\,\cong\, T$ using the standard properties of $\operatorname{Ext}^1_\Z$.
\end{proof}

\begin{proof}[Proof of Theorem~\textup{\ref{thm:obstructions}}]
We claim that each $v=\delta(u)$ is torsion, where $u\,=\,[P]\in H^1(X,\,C)$ corresponds to some $C$"-bundle $P$. There 
is a 
map $f\,\colon \,X\,\longrightarrow\, BC$ with $P\cong f^*[EC]$.
By using naturality we have a commutative diagram
\[\xymatrix{
	H^1(X,\,C)\ar[r]^-\delta & H^2(X,\,\C^*)\\
	H^1(BC,\,C)\ar[r]_-\delta\ar[u]^{f^*} & H^2(BC,\,\C^*)\ar[u]_{f^*}
}\]
It follows that $v\,=\,\delta(u)\,=\,\delta f^*[EC] \,=\, f^*\delta[EC]$ factors through the torsion group 
$H^2(BC,\,\C^*)\,\cong\, H^3(BC;\,\Z)$ (exponential sequence and Lemma~\ref{lem:all-torsion}).
\end{proof}

\begin{remark}\label{rem:infinite-dim}
When we allow infinite-dimensional Lie groups, the situation is drastically different.
For a separable complex Hilbert space $\mathcal{H}$, Kuiper's theorem asserts that the bounded invertible operators $GL(\mathcal{H})$ are contractible. Consider
\[
	1 \longrightarrow \C^* \longrightarrow GL(\mathcal{H}) \longrightarrow PGL(\mathcal{H}) \longrightarrow 1.
\]
Then the long exact sequence of homotopy groups shows that $PGL(\mathcal{H})=K(\Z,2)$ and  $BPGL(\mathcal{H})=K(\Z,3)$. By the exponential sequence we have a diagram
\[\xymatrix{
H^1(X,\,PGL(\mathcal{H}))\ar[r]^-\delta\ar[d] & H^2(X,\,\C^*)\ar[d]^\cong\\
[X,BPGL(\mathcal{H})]\ar[r] & H^3(X;\,\Z)
}\]
with lower horizontal isomorphism. Hence $\delta$ is surjective.
\end{remark}
%
%

\bibliographystyle{amsplain}

\end{document}